\documentclass[letterpaper,11pt,leqno]{article}
\usepackage{amsmath,amssymb,amsfonts}%

\usepackage{paper}

\usepackage{graphicx}%
\usepackage{multirow}%
\usepackage{amsthm}%
\usepackage{mathrsfs}%
\usepackage[title]{appendix}%
\usepackage{xcolor}%
\usepackage{textcomp}%
\usepackage{manyfoot}%
\usepackage{booktabs}%
\usepackage{algorithm}%
\usepackage{algorithmicx}%
\usepackage{algpseudocode}%
\usepackage{listings}%
\usepackage{causets}
\usepackage{doi}
\usepackage{bookmark}
\hypersetup{pdftitle={The Gold Partition Conjecture and the Lexicographic sum of posets}}

\newtheorem{example}{Example}%

\newtheorem{conjecture}[theorem]{Conjecture}

\usepackage{mathtools}
\def\multiset#1#2{\ensuremath{\left(\kern-.3em\left(\genfrac{}{}{0pt}{}{#1}{#2}\right)\kern-.3em\right)}}

\newtheorem*{lemma*}{Lemma}
\newtheorem*{theorem*}{Theorem}

\newcommand{\chain}[1][n]{\langle #1\rangle}

\begin{document}

\title{The Gold Partition Conjecture and the Lexicographic sum of posets}

\author{
Eric Dolores-Cuenca, Aldo Guzm\'an-S\'aenz, and Sangil Kim
%
\thanks{E. Dolores Cuenca: Department of Mathematics, Yonsei Univeristy 50 Yonsei-Ro, {Seodaemun-Gu}, {03722}, {Seoul}, South Korea\\ Department of Mathematics, Kyungpook National University 80 Daehak-ro, {Buk-gu}, {41566}, Daegu, South Korea.\\A. Guzm\'an-S\'aenz: Thomas J. Watson Research Center, IBM. 1101 Kitchawan Rd, {Yorktown Heights}, {10598}, New York, USA.\\
S. Kim: Department of Mathematics \& Institute for Future Earth, Pusan National University. 2, Busandaehak-ro 63beon-gil,
{Geumjeong-gu}, {46241}, Busan, South Korea.\\
We thank Professor Thomas Trotter and Professor Marcin Peczarski for their suggestions. We thank the reviewers for their feedback.\\
This work was supported by the National Research Foundation of Korea (NRF) grant funded by the Korea government(MSIT) (No. 2022R1A5A1033624) and Global - Learning \& Academic research institution for Master’s · Ph.D. students, and Postdocs (LAMP) Program of the National Research Foundation of Korea (NRF) grant funded by the Ministry of Education (No. RS-2023- 00301938). The work of E. Dolores Cuenca was supported by the Korea National Research Foundation (NRF) grant funded by the Korean government (MSIT) (RS-2025-00517727).}}

\date{\today}   


\begin{titlepage}
\maketitle

We prove that if a finite poset $Q$ satisfies the Gold Partition Conjecture, and $P$ is a finite poset, then for any $i\in P$ the lexicographic sum $P\circ_iQ$ of $P$ with $Q$ on the point $i$, satisfies the Gold Partition Conjecture.

Let $L(P)$ be the set of linearizations of $P$, $e(P)=\#L(P)$, $ \mathbb{P}(x<y)=\frac{\#\{f\in L(P)|x<_f y\}|}{e(P)}$, and $\delta(P)=\max_{x,y\in P}\min\{\mathbb{P}(x<y),$ $\mathbb{P}(y<x)\}$. We describe the behavior of those numbers when $P$ is a lexicographic sum. Namely, $\prod_{i=1}^n e(Q_i)\,$ divides $\,e(P( Q_1,\dots,Q_{|P|}))$, and  $\delta(P(Q_1,\dots,Q_n))\geq \max_{1\leq i\leq n}\{ \delta(Q_i)\}$. 

\end{titlepage}

\section{Introduction}\label{sec1}
A finite partially ordered set (poset) $P=(X,\leq_P)$ consist of a finite set $X$ and a partial order $\leq_P$. If $x\leq_P y$ and $x\neq y$ then we write $x<_Py$. 
Given a poset $P=(\{p_1,\dots,p_n\},\leq_P)$, we denote by   $|P|$ its cardinality, and by $L(P)$ the set of  linear orders on $P$ compatible with $\leq_P$.  We call such linear orders \emph{linearizations of $P$} and define $e(P)=|L(P)|$. By abuse of notation, we denote by $\langle n \rangle$ the \emph{$n$-chain} $(\{x_1,\dots,x_n,\}\leq)$ satisfying $x_1<x_2<\dots<x_n$.

The following conjecture was proposed by  Kislitsyn \cite{original}, Fredman \cite{original2}, and Linial \cite{original3}.

\begin{conjecture}[1/3-2/3 Conjecture]
For every finite poset $P$, that is not a linear order, there is a pair of points $x,y$ such that 
\begin{align*}
  \frac{|\{f\in L(P)\colon x<_f y\}|}{e(P)}\in[1/3,2/3].  
\end{align*}
\end{conjecture}
The paper~\cite{xprod} states that the 1/3-2/3 Conjecture is one of the most intriguing problems in the combinatorial theory of posets. For more information on the conjecture see~\cite{history, prob, diagrams,linext}.

The conjecture has been proven for the following families: posets of width 2~\cite{original3}, posets with a non trivial automorphism~\cite{symmetry}, semiorders~\cite{semiorders}, height 2 posets~\cite{Height2},  5-thin posets~\cite{five}, posets containing at most 11 points~\cite{gold}, 6-thin posets~\cite{6thin}, series-parallel posets~\cite{Nfree},  
 Young diagrams~\cite{diagrams},  and posets whose cover graph is a forest~\cite{graphforest}. The conjecture has been reformulated in terms of Coxeter groups~\cite{groups}, and in terms of similarity spaces~\cite{sim}.
 
 Marcin Peczarski~\cite{gold} proposed the following conjecture, which implies the 1/3-2/3 Conjecture: \begin{conjecture}[Gold Partition Conjecture (GPC)]
For every poset $P$, which is not a chain, there are two consecutive comparisons such that regardless of their results, the following inequality holds \begin{align*}
    t_0\geq t_1+t_2,
\end{align*} where $t_0=e(P)$,  $t_1$ is the number of linearizations of the poset after the first comparison, and $t_2$ is the number of linearizations after both comparisons. 
\end{conjecture}

 Let $C(P)$ be the number of comparisons required and sufficient to sort $P$, and consider $\phi=\frac{1+\sqrt{5}}{2}$ the golden ratio. If a poset $P$ satisfies the GPC, then $C(P)\leq {\log_\phi{e(P)}}$~\cite{gold}. The name of the conjecture follows from the role of the golden ratio $\phi=\frac{1+\sqrt{5}}{2}$ in the bound.

 The GPC is known to hold for posets of width 2, semiorders, posets containing at most 11 elements~\cite{gold}, 6 thin posets~\cite{6thin}, posets containing non trivial automorphisms~\cite{CGPC}, and a class containing N free posets~\cite{CGPC}.

\begin{definition} Given posets $(Q_1,\leq_{Q_1}),\dots,(Q_n,\leq_{Q_n}), P=(\{p_1,\dots,p_n\},\leq_P)$,
the lexicographic sum $P(Q_1,\dots,Q_n)$ is defined by the set $\sqcup Q_i$ with order 

\begin{align*}
  x\leq_{P(Q_1,\dots,Q_n)}y \hbox{ if }
\begin{cases}
x,y\in Q_j &\hbox{ and } x\leq_{Q_j}y,\\
\hbox{or } x\in Q_i, y\in Q_j, i\neq j &\hbox{ and }p_i<_{P}p_j.
\end{cases}  
\end{align*}

If all $Q_j$ are one point, except $Q_i=Q$, then we denote the lexicographic sum by $P\circ_i Q$. 
\end{definition}

\begin{definition}
A Hasse diagram of a poset $P$ is a graph in which we draw an edge from a vertex $a\in P$ to a vertex $b\in P$, where $b$ is situated above $a$, if $b$ is a successor of $a$ under $\leq_P$.    
\end{definition}

For example, a Hasse diagram for $\{x<y<z,x<w\}$ is \pcauset{1,4,2,3}. The Hasse diagram of $P\circ_iQ$ can be drawn by replacing the point $i$ of $P$ by the poset $Q$, and adding those edges from the antecessors of $i$ to the minimum elements of $Q$, and from the maximum elements of $Q$ to the successors of $i$.

 The lexicographic sum has been studied extensively~\cite{lex,dim, notation, dt}. The lexicographic sum assigns a poset $P$ to an endomorphism of posets, see also~\cite[Figure 9]{Shuffleseries} collecting examples of sets $A$ such that every poset defines an endomorphim of A.

\begin{definition}
The lexicographic sum $P(Q_1,\dots,Q_n)$ is said to be trivial if $P=\langle 1 \rangle$ or if all $Q_i= \langle 1 \rangle$. A poset is called decomposable if it admits at least one non trivial lexicographic sum. Otherwise the poset is called indecomposable.  
\end{definition}
 For example, given a set $X$ the poset $2^X\setminus\{\varnothing, X\}$ is an indecomposable poset with non trivial automorphisms~\cite{notation}.

Our main contribution is Theorem~\ref{Lemma:main}, which shows that the GPC is preserved under lexicographic sum, more precisely, if one of the inputs $Q_i$ satisfy the GPC then the lexicographic sum does too.

\section{The space of linearizations of a poset}

By abuse of notation, we denote by $f$ both a linear order on a poset P and the corresponding order-preserving bijective labeling map $f\colon P\rightarrow [1,|P|]$.

\label{Sec:main}

Let $P$ be a finite poset and $i\in P$. From the point of view of $i$, there are three types of points in $P$: the set $P_{>i}$ containing those points that are bigger than $i$, the set $P_{<i}$ containing those points that are smaller than $i$ and the set $P_{|| i}$ consisting of points incomparable to $i$.

\begin{remark}[Locality of lexicographic sum]\label{rm}
Note that in any linearization $f$ of $P\circ_i Q$, for all $q\in Q$ we have $f(j)>f(q)$ if $j\in P_{>i}$ and $f(k)<f(q)$ if $k\in P_{<i}$.
\end{remark}

To understand how a linearization $f$ of $P\circ_i Q$ maps the points of $P_{|| i}$ we consider two pieces of information: the poset structure of $P$, and the possible image of $Q$ under $f$. 

\begin{definition}
Given  $f\in L(P\circ_i Q)$,  we denote by $f|_Q$ the linear order of $Q$ defined by $q_1\leq_{f|_Q}q_2$ iff $f(q_1)\leq f(q_2)$.
    
\end{definition}

We now show that it is possible to index the linearizations of $P\circ_i Q$ with columns $L(Q)$ and rows $A$ for some non unique set $A$.

\begin{lemma}\label{Lemma:loctoprod}
Let $P, Q$ be finite posets. If we fix $g\in L(Q)$ and let $R_g=\{f\in L(P\circ_i Q)\colon g=f|_Q\}$. Then, we have the set isomorphism $L(P\circ_i Q)\simeq R_g\times L(Q)$.
\end{lemma}
\begin{proof}
The main idea is to decompose linearizations of $P\circ_i Q$ in two parts, the part acting on $Q$ and the part acting on $P\setminus \{i\}$.  The restriction give us a map $\pi\colon L(P\circ_i Q)\rightarrow L(Q)$, we now proceed to study the fibers of $\pi$ and show they all have the same number of elements. 

For every $f\in R_g$, consider the image of $Q$ under $f\colon \{a_1,\dots,a_l \}\subset [1,l+m-1]$, where $l=|Q|, m=|P|$. Then, given $h\in L(Q)$ consider $\sigma$ the isomorphism between $g(Q)$ and $h(Q)$. We define a linearization $f\circ_i h$ of $P\circ_i Q$ by sending $q\in Q$ to $a_{\sigma(q)}$ and every $p\in P\setminus{\{i\}}$ to $f\circ_i h(p) \colon= f(p)$. Denote by $f\times L(Q)$ the set of such linearizations constructed with elements from $L(Q)$. 
 Then $\sqcup_{f\in R_g} f\times L(Q)\subset L(P\circ_i Q)$. 

Now given $h\in L(P\circ_i Q)$ there is an $f\in R_g$ such that $h\in f\times L(Q)$, namely, $f$ is the linearization that coincides with $h$ on $P\setminus \{i\}$.

 Then $ L(P\circ_i Q)\subset \sqcup_{f\in R_g} f\times L(Q)$.

\end{proof}

\begin{example}\label{ex:useful}

Let $N=\{w<y>x<z\}$. Consider $N\circ_w(\{r\}\sqcup \{t<u\})$. We describe explicitly the linearizations of this poset in 
Table~\ref{table:3}.

\begin{table}[htb]
\caption{Table containing every element of $L(N\circ_w(\{r\}\sqcup \{t<u\}))$. Each linearization is described by showing the values of each variable. Columns are indexed by elements of $L(\{r\}\sqcup \{t<u\})$.\label{table:3}}
\begin{tabular}{| c |c| c| c| c| c|| c |c| c| c| c| c|| c |c| c| c| c| c|}
\hline
\multicolumn{18}{|c|}{$L(N\circ_w(\{r\}\sqcup \{t<u\}))$} \\
\hline
\multicolumn{6}{|c||}{$\{r<t<u\}$}&
\multicolumn{6}{|c||}{$\{t<r<u\}$}&
\multicolumn{6}{|c|}{$\{t<u<r\}$}\\
\hline
\hline
 r & t & u& x & y & z  &
 r & t & u& x & y & z  &
 r & t & u& x & y & z  \\
 
\hline
\hline
\hline
 1 & 2 & 3 & 4 & 5 & 6 &
 2 & 1 & 3 & 4 & 5 & 6 &
 3 & 1 & 2 & 4 & 5 & 6 \\  
\hline
\hline
 1 & 2 & 3 & 4 & 6 & 5 &
 2 & 1 & 3 & 4 & 6 & 5 &
 3 & 1 & 2 & 4 & 6 & 5 \\

\hline 
\hline
 1 & 2 & 4 & 3 & 5 & 6 &
 2 & 1 & 4 & 3 & 5 & 6 &
 4 & 1 & 2 & 3 & 5 & 6 \\  
\hline
\hline

 1 & 2 & 4 & 3 & 6 & 5 &
 2 & 1 & 4 & 3 & 6 & 5 &
 4 & 1 & 2 & 3 & 6 & 5 \\  
\hline
\hline
 1 & 3 & 4 & 2 & 5 & 6 &
 3 & 1 & 4 & 2 & 5 & 6 &
 4 & 1 & 3 & 2 & 5 & 6 \\  
\hline
\hline

 1 & 3 & 4 & 2 & 6 & 5 &
 3 & 1 & 4 & 2 & 6 & 5 &
 4 & 1 & 3 & 2 & 6 & 5 \\  
 \hline
\hline
 2 & 3 & 4 & 1 & 5 & 6 &
 3 & 2 & 4 & 1 & 5 & 6 &
 4 & 2 & 3 & 1 & 5 & 6 \\  
\hline
\hline

 2 & 3 & 4 & 1 & 6 & 5 &
 3 & 2 & 4 & 1 & 6 & 5 &
 4 & 2 & 3 & 1 & 6 & 5\\
\hline
\hline
 1 & 2 & 5 & 3 & 6 & 4 &
 2 & 1 & 5 & 3 & 6 & 4 &
 5 & 1 & 2 & 3 & 6 & 4 \\  
\hline
\hline
 1 & 3 & 5 & 2 & 6 & 4 &
 3 & 1 & 5 & 2 & 6 & 4 &
 5 & 1 & 3 & 2 & 6 & 4 \\  
\hline
\hline
 2 & 3 & 5 & 1 & 6 & 4 &
 3 & 2 & 5 & 1 & 6 & 4 &
 5 & 2 & 3 & 1 & 6 & 4 \\  
 
\hline
\hline
 1 & 4 & 5 & 2 & 6 & 3 &
 4 & 1 & 5 & 2 & 6 & 3 &
 5 & 1 & 4 & 2 & 6 & 3 \\  
\hline
\hline
 2 & 4 & 5 & 1 & 6 & 3 &
 4 & 2 & 5 & 1 & 6 & 3 &
 5 & 2 & 4 & 1 & 6 & 3 \\  

\hline
\hline
 3 & 4 & 5 & 1 & 6 & 2 &
 4 & 3 & 5 & 1 & 6 & 2 &
 5 & 3 & 4 & 1 & 6 & 2 \\
\hline
\end{tabular}
\end{table}
\end{example}
 We do not have much information about the fiber $R_g$, but the situation clarifies when we consider all inputs of a lexicographic sum.
\begin{lemma}\label{lemma:cubito}
    Let $P$ be a poset with $n=|P|$ and let $Q_1,\dots,Q_n$ be posets. Then we can organize $L(P(Q_1,\dots,Q_n))$  into a $n+1$ dimensional array $A$. The dimension $i, i\leq n$ is indexed by the linearizations of $Q_i$. The $n+1$ dimension is indexed by linearizations of $P( \langle|Q_1|\rangle,\dots,\langle|Q_{n}|\rangle)$.
    \end{lemma}\begin{proof}Note that the last dimension $n+1$ depends only on the number of elements of each $\{Q_i\}$ and not on the posets $\{Q_i\}$. We choose a linearization of $P( \langle|Q_1|\rangle,\dots,\langle|Q_{n}|\rangle)$, and then for every $Q_i$ we replace the chain $\langle|Q_i|\rangle$ by a linearization of $Q_i$, this process returns a linearization of the lexicographic sum.  Conversely, by restriction,
    every linearization of $P( Q_1,\dots,Q_{n})$ gives us a unique pair consisting of a linearization of  $P( \langle|Q_1|\rangle,\dots,\langle|Q_{n}|\rangle)$ and  linearizations of each input $Q_i$.
\end{proof}
\begin{remark} We provide a combinatorial interpretation of the set   $L\left(P( \langle|Q_1|\rangle,\dots,\langle|Q_{n}|\rangle)\right)$ from the theory of multiset species~\cite{Multiset}.   
 Assume the poset $P$ is a disjoint union of points $\{x_1,\dots,x_n\}$. Given $\{Q_i\}$ we construct a multiset $M$ with  $|Q_i|$ copies of the letter $x_i$. Then, each linearization of  $P( \langle|Q_1|\rangle,\dots,\langle|Q_{n}|\rangle)$ corresponds to a word in the alphabet $\{x_1,\dots,x_n\}$ where each letter $x_i$ occurs $|Q_i|$ times. This are elements of
 $Sy(L)[M]$ in the notation of~\cite{Multiset}. If we now consider an arbitrary poset $P$, we understand the order relations of $P$ as restrictions on the order of the letters in the alphabet, and then, restrictions on the words we can produce. For example, if $P=\{x,y,z|x<y\}$ and we consider the linearizations of $P(\langle|2|\rangle,\langle|3|\rangle,\langle|2|\rangle)$ (assuming the first input corresponds to $x$, the second to $y$, and the last to $z$), then  one possible word associated to the linearizations is $xzxyyzy$. Note that every letter $x$ appears before the letter $y$.    
\end{remark}

From Lemma~\ref{Lemma:loctoprod} it follows that $e(Q)|e(P\circ_j Q)$. We now consider more inputs $Q_i$.
\begin{corollary}\label{Lemma:cube}
Let $P$ be a poset with $|P|=n$ and $Q_1,\dots,Q_n$ finite posets. Then 
$\prod_{i=1}^n e(Q_i)\,$ divides $\,e(P( Q_1,\dots,Q_{n}))$. More precisely, 
\begin{align*}
  e\big(P( \langle|Q_1|\rangle,\dots,\langle|Q_{n}|\rangle)\big)=\frac{e\big(P( Q_1,\dots,Q_{n})\big)}{\prod_{i=1}^n e(Q_i)}.  
\end{align*}
\end{corollary}
\begin{proof}
    Follows from the proof of Lemma~\ref{lemma:cubito}.
\end{proof}

\begin{theorem}If $Q$ is a finite poset that satisfies the GPC, and $P$ is a finite poset, then for every $i\in P,$ the lexicographic sum $P\circ_iQ$ satisfies the GPC.
\label{Lemma:main}
\end{theorem}
\begin{proof} 
    If $Q$ satisfies the GPC then fix the comparisons added to $Q$ with $t_0=e(Q), t_1, t_2$  $t_0\geq t_1+t_2$. Then, we use the same two comparisons of $Q$ on $P\circ_i Q$ to obtain $t^\prime_{0},t^\prime_{1},t^\prime_{2}.$
    If we arrange the linearizations of $P\circ_i Q$ using the linearizations of $Q$ as columns, and we denote by $k$ be the height of the table, then $t^\prime_{0}=kt_{0},t^\prime_{1}=kt_{1},t^\prime_{2}=kt_{2}$ and $t^\prime_0\geq t^\prime_1+t^\prime_2$. 
\end{proof}

 \begin{example}\label{Ex:counter}The following poset has width 8, it is not a semiorder since it contains $\pcauset{4,1,2,3}$, and the poset has 19 elements: 
     $\pcauset{1,15, 13,17,18,16,12,10,14,8,11,9,6,7,0,5,2,3,4}$.
     Since the poset can be written as a lexicographic sum on the green vertex $\pcauset[subset A/.style={green}]{1,15, 13,16/subset A,12,10,14,8,11,9,6,7,0,5,2,3,4}\circ\pcauset{3,1,2}$ and the poset $\pcauset{3,1,2}$ satisfies the GPC, then so does the original poset. Note that this poset does not belong to the classes of posets known in the current literature to satisfy the GPC. 
 \end{example}

    Let $\delta(P)=\max_{x,y\in P}\min\{\mathbb{P}(x<y),\mathbb{P}(y<x)\}$, where $\mathbb{P}(x<y)=\frac{|\{f\in L(P)\colon x<_f y\}|}{e(P)}$ is the probability that a linearization $f$ of the poset $P$ will satisfy $f(x)<f(y)$. In~\cite{w2}, the corresponding authors show that $\delta(P\oplus Q)\colon=\delta(\{x<y\}(P,Q))= \max\{\delta(P),\delta(Q)\}$, and $\delta(P\sqcup Q)\colon=\delta(\{x,y\}(P,Q))\geq \max\{\delta(P),\delta(Q)\}$.
    
    \begin{corollary}\label{delta} Let $P, |P|=n$ and $Q_1,\dots,Q_n$ be finite posets. Then $\delta(P(Q_1,\dots,Q_n))\geq \max_{1\leq i\leq n}\{ \delta(Q_i)\}$.     
    \end{corollary}
\begin{proof}
    We construct a table whose columns are linearizations of $Q_1$.
    Then if $x_1,y_1\in Q_1$ are such that $\delta(Q_1)=\mathbb{P}(x_1<_{Q_1}y_1)$, then, as in the proof of Lemma~\ref{Lemma:main}, the ratio of linearizations in $P(Q_1,\dots,Q_n)$ where $f(x_1)<f(y_1)$ is the same as in $Q_1$, that is:
    \begin{align*}
      \delta(P(Q_1,\dots,Q_n))\geq \mathbb{P}(x_1<_{P(Q_1,\dots,Q_n)}y_1)=\mathbb{P}(x_1<_{Q_1}y_1)=\delta(Q_1).  
    \end{align*}
    Note that this argument works for an arbitrary $Q_i$.
\end{proof}

\begin{remark}
The paper \cite{diagrams} asks if given a poset $R$, we can find a poset $Q$ with smaller width and $\delta(Q)<\delta(R)$. Let $R$ be a poset with non trivial lexicographic sum $R=P(Q_1,\dots,Q_s)$, and assume we have $\delta(Q_i)<\delta(Q_j)$ for some $i,j$ non linear posets, and width$(Q_i)<$ width$(R)$. Then by Corollary~\ref{delta} $\delta(Q_i)<\delta(R)$.    
\end{remark}

    There is a well established concept of dimension of a poset~\cite{dt}, and Hiraguchi's relation says that $dim(P(Q_1,\dots,Q_n))=\max\{dim(P), dim(Q_1),\dots,dim(Q_n)\}$. One could wonder if $\delta(P(Q_1,\dots,Q_n))\geq\delta(P)$, for example $\delta(\pcauset{1,4,2,5,3}\circ_a \pcauset{3,1,2})\geq \delta(\pcauset{1,4,2,5,3})$ for every choice of $a\in \pcauset{1,4,2,5,3}$.

\begin{example}
The following counter example was found by Anderson Vera and the first author:
\begin{align*}
\delta(\pcauset{1,6,3,4,2,5})=7/15<.5=\delta(\pcauset{1,5,3,2,4}).  
\end{align*}

We describe the construction of the counter example, since the method can be useful for other similar problems.
Let $Q=\chain[n]$.
 We group the linearizations of $P\circ_{x_i}Q$ in the following way.
 To every $f\in L(P)$ consider the points $b = \min\{ f(r)\colon r\geq_P x_i\}$ and $c=\max\{f(s)\colon s\leq_P x_i\}$.
 Then in $P$ there are some points $t_1,\dots,t_k$ such that
 \begin{align*}
     c\leq f(t_1)\leq\dots\leq f(x_i)\leq \dots\leq f(t_k)\leq b.
 \end{align*}
 We define $T_f$ to contain  all the $k+1$ linearizations of $P$ determined by varying the location of $f(x_i)$ between $c$ and $b$.
 
 On $P\circ_{x_i}Q$, we consider those linearizations whose projection to $P$ is in $T_f$. Here the projection $\tilde{n}$ sends $f(Q)$ to the minimum of $f(Q)$.
 The multiset coefficient is defined by  $\multiset{y}{x}\colon={y+x-1\choose x}$.
 Then, $\tilde{n}^{-1}(T_f)$ has $\multiset{k+1}{|Q|}$ elements, since $Q$ is a linear poset. One shows that this procedure partitions the linearizations of $P\circ_{x_i}Q$ using the partition of the linearizations of $P$.

Now, for $x,y\in P\setminus\{ x_i\}$, we first split the linearizations of $P$ into two sets: a) the set of linearizations in which $f(x)<f(y)$, b) the set of linearizations in which  $f(y)<f(x)$.
We then split each set based on the cardinality of the corresponding $T_f$, let $c_i$ be the number of sets of a) in which $T_f$ has cardinality $i+1$. Let $d_i$ be the number of sets in b) in which $T_f$ has cardinality $i+1$. We compute 
\begin{align*}
  \mathbb{P}(x<_{P}y)=\frac{\sum_{i=1}^n c_i(i+1)}{\sum_{i=1}^n (c_i+d_i)(i+1) }  
\end{align*}
 and
 \begin{align*}
   \mathbb{P}(x<_{P\circ_{x_i}(Q)}y)=\frac{\sum_{i=1}^n c_i\multiset{i+1}{|Q|}}{\sum_{i=1}^n (c_i+d_i)\multiset{i+1}{|Q|} }.  
 \end{align*}
 
By studying the values $c_i,d_i$ one can discover posets with $\delta(P\circ_{x_i} Q)<\delta(P)$. 
\end{example}

\begin{corollary}
    If $Q_i$ satisfies the ${1}/{3}-{2}/{3}$ Conjecture, so does $P(Q_1,\dots,Q_n)$ for $P$ poset with $n$ points and $Q_1,\dots,Q_{i-1},Q_{i+1},\dots, Q_n$ arbitrary posets. 

\end{corollary}
    \begin{proof}
        Follows from Corollary~\ref{delta}.
    \end{proof}

In Example~\ref{ex:useful}, by examining the columns and the values of $r,t$ we compute 
\begin{align*}
  \frac{\#\{f\in L(N\circ_w(\{r\}\sqcup \{t<u\}))\colon f(t)<f(r)\}}{e(N\circ_w(\{r\}\sqcup \{t<u\}))}= 1/3.  
\end{align*}

\section{Final remarks}

One way to speed up the verification that $P(Q_1,\dots,Q_n)$ satisfy the GPC is to first test if any non linear order out of $Q_1,\dots, Q_n$ (which are smaller posets) satisfies the GPC. Another point of view is that if we show that every indecomposable poset satisfies the GPC, and for indecomposable $P$ and positive integers $ \{s_i\}$ such that $\prod s_i>1$ then $P(\chain[s_1],\dots,\chain[s_n])$ satisfies the GPC, it follows that all posets  satisfy the GPC.
See~\cite[Section 7.2]{notation} for details about a canonical decomposition of posets as lexicographic sums.

Lemma~\ref{Lemma:main} increases ``exponentially'' the families of posets that satisfy the 1/3-2/3 Conjecture by including any poset of the form $P\circ_i Q$ with $Q$ in the list of families mentioned on the introduction. The poset in Example~\ref{Ex:counter} satisfies the 1/3-2/3 Conjecture and it is not in any of the classes previously known to satisfy the conjecture. 
 From the point of view of sorting, our results are consistent with the use of divide and conquer strategies~\cite{intr} on non trivial lexicographic sums. 

\bibliography{biblio}

@Article{original,
 Author = {Kislitsyn, S. S.},
 Title = {A Finite partially ordered sets and their corresponding permutation sets},
 FJournal = {Mathematical Notes},
 Journal = {Math. Notes},
 ISSN = {0001-4346},
 Volume = {4},
 Pages = {798--801},
 Year = {1969},
 Language = {English},
 DOI = {10.1007/BF01111312},
 zbMATH = {3363564},
 Zbl = {0229.05016}
}

@Article{original2,
 Author = {Fredman, Michael L.},
 Title = {How good is the information theory bound in sorting?},
 FJournal = {Theoretical Computer Science},
 Journal = {Theor. Comput. Sci.},
 ISSN = {0304-3975},
 Volume = {1},
 Pages = {355--361},
 Year = {1976},
 Language = {English},
 DOI = {10.1016/0304-3975(76)90078-5},
 zbMATH = {3513338},
 Zbl = {0327.68056}
}

@Article{original3,
 Author = {Linial, Nathan},
 Title = {The information-theoretic bound is good for merging},
 FJournal = {SIAM Journal on Computing},
 Journal = {SIAM J. Comput.},
 ISSN = {0097-5397},
 Volume = {13},
 Pages = {795--801},
 Year = {1984},
 Language = {English},
 DOI = {10.1137/0213049},
 zbMATH = {3874634},
 Zbl = {0548.68065}
}

@Article{w2,
 Author = {Sah, Ashwin},
 Title = {Improving the {{\(\frac{1}{3}\)}}-{{\(\frac{2}{3}\)}} conjecture for width two posets},
 FJournal = {Combinatorica},
 Journal = {Combinatorica},
 ISSN = {0209-9683},
 Volume = {41},
 Number = {1},
 Pages = {99--126},
 Year = {2021},
 Language = {English},
 DOI = {10.1007/s00493-020-4091-3},
 zbMATH = {7364752},
 Zbl = {1474.06006}
}

@Article{semiorders,
 Author = {Brightwell, Graham R.},
 Title = {Semiorders and the 1/3-2/3 conjecture},
 FJournal = {Order},
 Journal = {Order},
 ISSN = {0167-8094},
 Volume = {5},
 Number = {4},
 Pages = {369--380},
 Year = {1989},
 Language = {English},
 DOI = {10.1007/BF00353656},
 zbMATH = {4114702},
 Zbl = {0681.06001}
}

@Article{Nfree,
 Author = {Zaguia, Imed},
 Title = {The {{\(1/3\)}}-{{\(2/3\)}} conjecture for {{\(N\)}}-free ordered sets},
 FJournal = {The Electronic Journal of Combinatorics},
 Journal = {Electron. J. Comb.},
 ISSN = {1077-8926},
 Volume = {19},
 Number = {2},
 Pages = {research paper p29, 5},
 Year = {2012},
 Language = {English},
 zbMATH = {6045722},
 Zbl = {1288.06005},
 url={https://doi.org/10.37236/2345 },
 doi={10.37236/2345}
}

@Article{graphforest,
 Author = {Zaguia, Imed},
 Title = {The 1/3-2/3 conjecture for ordered sets whose cover graph is a forest},
 FJournal = {Order},
 Journal = {Order},
 ISSN = {0167-8094},
 Volume = {36},
 Number = {2},
 Pages = {335--347},
 Year = {2019},
 Language = {English},
 DOI = {10.1007/s11083-018-9469-0},
 zbMATH = {7083964},
 Zbl = {1443.06004}
}

@Article{Height2,
 Author = {Trotter, William T. and Gehrlein, William and Fishburn, Peter C.},
 Title = {Balance theorems for height-2 posets},
 FJournal = {Order},
 Journal = {Order},
 ISSN = {0167-8094},
 Volume = {9},
 Number = {1},
 Pages = {43--53},
 Year = {1992},
 Language = {English},
 DOI = {10.1007/BF00419038},
 url={https://doi.org/10.1007/BF00419038},
 zbMATH = {146915},
 Zbl = {0769.06004}
}

@Article{symmetry,
 Author = {Ganter, Bernhard and H{\"a}fner, Gerhart and Poguntke, Werner},
 Title = {On linear extensions of ordered sets with a symmetry},
 FJournal = {Discrete Mathematics},
 Journal = {Discrete Math.},
 ISSN = {0012-365X},
 Volume = {63},
 Pages = {153--156},
 Year = {1987},
 Language = {English},
 DOI = {10.1016/0012-365X(87)90005-7},
 url={https://doi.org/10.1016/0012-365X(87)90005-7},
 zbMATH = {3981228},
 Zbl = {0607.06001}
}

@Article{diagrams,
 Author = {Olson, Emily J. and Sagan, Bruce E.},
 Title = {On the {{\(1/3-2/3\)}} conjecture},
 FJournal = {Order},
 Journal = {Order},
 ISSN = {0167-8094},
 Volume = {35},
 Number = {3},
 Pages = {581--596},
 Year = {2018},
 Language = {English},
 DOI = {10.1007/s11083-017-9450-3},
 url={https://doi.org/10.1007/s11083-017-9450-3},
 zbMATH = {6949483}
}

@Book{notation,
 Author = {Schr{\"o}der, Bernd},
 Title = {Ordered sets. {An} introduction with connections from combinatorics to topology},
 Edition = {2},
 ISBN = {978-3-319-29786-6; 978-3-319-29788-0},
 Year = {2016},
address ={Basel},
 Publisher = {Birkh{\"a}user/Springer},
 DOI = {10.1007/978-3-319-29788-0}, 
}

@Article{6thin,
 Author = {Peczarski,  Marcin Piotr},
 Title = {The gold partition conjecture for 6-thin posets},
 FJournal = {Order},
 Journal = {Order},
 ISSN = {0167-8094},
 Volume = {25},
 Number = {2},
 Pages = {91--103},
 Year = {2008},
 Language = {English},
 DOI = {10.1007/s11083-008-9081-9},
 zbMATH = {5323314},
 Zbl = {1155.06004}
}

@Article{gold,
 Author = {Peczarski,  Marcin Piotr},
 Title = {The gold partition conjecture},
 FJournal = {Order},
 Journal = {Order},
 ISSN = {0167-8094},
 Volume = {23},
 Number = {1},
 Pages = {89--95},
 Year = {2006},
 Language = {English},
 DOI = {10.1007/s11083-006-9033-1},
 url={https://doi.org/10.1007/s11083-006-9033-1},
 zbMATH = {5053998},
 Zbl = {1099.06001}
}

@Article{xprod,
 Author = {Brightwell, Graham R. and Felsner, S. and Trotter, William. T.},
 Title = {Balancing pairs and the cross product conjecture},
 FJournal = {Order},
 Journal = {Order},
 ISSN = {0167-8094},
 Volume = {12},
 Number = {4},
 Pages = {327--349},
 Year = {1995},
 Language = {English},
 DOI = {10.1007/BF01110378},
 url={https://doi.org/10.1007/BF01110378},
 zbMATH = {852894},
 Zbl = {0839.06002}
}

@Book{dt,
 Author = {Trotter, William T.},
 Title = {Combinatorics and partially ordered sets: dimension theory},
 Series = {Johns Hopkins Ser. Math. Sci.},
 ISBN = {0-8018-4425-8},
pages={328},
 Year = {1992},
 Publisher = {The Johns Hopkins University Press},
edition   = {1},  
address = {Baltimore}
}

@Article{prob,
 Author = {Saks, Michael},
 Title = {Unsolved problems},
 FJournal = {Order},
 Journal = {Order},
 ISSN = {0167-8094},
 Volume = {2},
 Number = {3},
 Pages = {327--330},
 Year = {1985},
 Language = {English},
 DOI = {10.1007/BF00333138},
 url={https://doi.org/10.1007/BF00333138},
}

@Article{five,
 Author = {Brightwell, Graham R. and Wright, Colin},
 Title = {The 1/3-2/3 conjecture for 5-thin posets},
 FJournal = {SIAM Journal on Discrete Mathematics},
 Journal = {SIAM J. Discrete Math.},
 ISSN = {0895-4801},
 Volume = {5},
 Number = {4},
 Pages = {467--474},
 Year = {1992},
 Language = {English},
 DOI = {10.1137/0405037},
 url={https://doi.org/10.1137/0405037},
 zbMATH = {147017},
 Zbl = {0777.06001}
}

@Article{lex,
 Author = {Hiraguchi, Toshio},
 Title = {On the dimension of partially ordered sets},
 FJournal = {The Science Reports of Kanazawa University},
 Journal = {Sci. Rep. Kanazawa Univ.},
 ISSN = {0022-8338},
 Volume = {1},
 Pages = {77--94},
 Year = {1951},
 Language = {English},
 zbMATH = {3318560},
 Zbl = {0200.00013}
}

@Article{dim,
 Author = {Hiraguchi, Toshio},
 Title = {On the dimension of orders},
 FJournal = {The Science Reports of Kanazawa University},
 Journal = {Sci. Rep. Kanazawa Univ.},
 ISSN = {0022-8338},
 Volume = {4},
 Number = {1},
 Pages = {1--20},
 Year = {1955},
 Language = {English},
 zbMATH = {3318593},
 Zbl = {0200.00046}
}

@Book{intr,
 Author = {Cormen, Thomas H. and Leiserson, Charles E. and Rivest, Ronald L. and Stein, Clifford},
 Title = {Introduction to algorithms.},
 Edition = {3rd ed.},
 ISBN = {978-0-262-03384-8; 978-0-262-53305-8},
 Year = {2009},
 Publisher = { MIT Press},
 Language = {English},
 zbMATH = {5645878},
 Zbl = {1187.68679},
address={Cambridge, MA}
}

@Article{history,
 Author = {Brightwell, Graham  R.},
 Title = {Balanced pairs in partial orders},
 FJournal = {Discrete Mathematics},
 Journal = {Discrete Math.},
 ISSN = {0012-365X},
 Volume = {201},
 Number = {1-3},
 Pages = {25--52},
 Year = {1999},
 Language = {English},
 DOI = {10.1016/S0012-365X(98)00311-2},
 url={https://doi.org/10.1016/S0012-365X(98)00311-2},
 zbMATH = {1334537},
 Zbl = {0940.06002}
}

@article{Shuffleseries,
 author = {Ahmad, Khushdil and Dolores-Cuenca, Eric Rubiel and Shabbir, Khurram},
 title = {Shuffle series},
 fjournal = {Journal of Algebraic Combinatorics},
 journal = {J. Algebr. Comb.},
 issn = {0925-9899},
 volume = {61},
 number = {2},
 pages = {40},
 note = {Id/No 19},
 year = {2025},
 language = {English},
 doi = {10.1007/s10801-025-01386-7},
 zbMATH = {8015153},
 Zbl = {1561.18015}
}

@Article{CGPC,
 Author = {Peczarski, Marcin Piotr},
 Title = {Comments on the golden partition conjecture},
 FJournal = {Contributions to Discrete Mathematics},
 Journal = {Contrib. Discrete Math.},
 ISSN = {1715-0868},
 Volume = {12},
 Number = {1},
 Pages = {106--109},
 Year = {2017},
 Language = {English},
 zbMATH = {6820576},
 Zbl = {1378.06001},
 url={https://doi.org/10.11575/cdm.v12i1.62349},
doi={10.11575/cdm.v12i1.62349}
}

@INPROCEEDINGS{sim,
author={Rozinek, Ondřej},
booktitle={2024 IEEE 17th International Scientific Conference on Informatics (Informatics)}, 
title={A Similarity Space Approach to the 1/3-2/3 Conjecture in Partially Ordered Sets}, 
year={2024},
volume={},
number={},
pages={516-521},
  doi={10.1109/Informatics62280.2024.10900886}}

@article{groups,
 author = {Gaetz, Christian and Gao, Yibo},
 title = {Balance constants for {Coxeter} groups},
 fjournal = {Transactions of the American Mathematical Society},
 journal = {Trans. Am. Math. Soc.},
 issn = {0002-9947},
 volume = {376},
 number = {12},
 pages = {8973--9000},
 year = {2023},
 language = {English},
 doi = {10.1090/tran/9040},
 zbMATH = {7772580}
}

@article{linext,
    author = {Chan, Swee Hong and Pak, Igor},
    title = {Linear extensions of finite posets},
    journal = {EMS Surv. Math. Sci.},
    year = 2025
}

@article{Multiset,
 author = {Mendez, Miguel},
 title = {Multisets and the combinatorics of symmetric functions},
 fjournal = {Advances in Mathematics},
 journal = {Adv. Math.},
 issn = {0001-8708},
 volume = {102},
 number = {1},
 pages = {95--125},
 year = {1993},
 language = {English},
 doi = {10.1006/aima.1993.1057},
 zbMATH = {554278},
 Zbl = {0793.05142}
}
\end{document}